\documentclass[11pt]{article}

\usepackage[english]{babel}
\usepackage{amsmath,amssymb,amsthm,amscd,enumerate}
\usepackage{graphicx,color}

\newtheorem{theorem}{Theorem}[section]

\newtheorem{lemma}[theorem]{Lemma}

\theoremstyle{definition}
\newtheorem{definition}[theorem]{Definition}

\newtheorem*{remark}{Remark}

\definecolor{comcolor}{rgb}{0.9,0.3,0.3}
\definecolor{starcolor}{rgb}{0.3,0.3,0.9}

\newcommand{\Z}{\mathbb{Z}}

\newcommand{\R}{\mathbb{R}}
\newcommand{\N}{\mathbb{N}}
\newcommand{\B}{\mathcal{B}}

\newcommand{\vi}{\mathbf{i}}
\newcommand{\vx}{\mathbf{x}}
\newcommand{\vt}{\mathbf{t}}

\renewcommand{\P}{\mathbb{P}}

\newcommand{\Pl}{\operatorname{\mathbf{P}}}
\newcommand{\E}{\operatorname{E}}

\newcommand{\Var}{\operatorname{Var}}
\newcommand{\F}{\mathcal{F}}

\newcommand{\expect}[1]{\E \left[ #1 \right]}
\newcommand{\indicate}[1]{\mathbf{1} \! \left \{ #1 \right \}}

\renewcommand{\wp}[1]{: \! #1 \! :}

\newcommand{\mZ}{\mathcal{Z}}

\renewcommand{\equiv}{\stackrel{(d)}{=}}

\begin{document}

\title{The Continuum Directed Random Polymer}

\author{Tom Alberts$^{\dag}$, Konstantin Khanin$^{\ddag}$, and Jeremy Quastel$^{\ddag}$ \footnote{Research of all three authors supported by the Natural Sciences and Engineering Research Council of Canada.} \\ \\ \small{$^{\dag}$California Institute of Technology \,\, \& \,\, $^{\ddag}$University of Toronto}}

\date{}

\maketitle

\begin{abstract}
Motivated by discrete directed polymers in one space and one time dimension, we construct a continuum directed random polymer that is modeled by a continuous path interacting with a space-time white noise. The strength of the interaction is determined by an inverse temperature parameter $\beta$, and for a given $\beta$ and realization of the noise the path evolves in a Markovian way. The transition probabilities are determined by solutions to the one-dimensional stochastic heat equation. We show that for all $\beta > 0$ and for almost all realizations of the white noise the path measure has the same H\"{o}lder continuity and quadratic variation properties as Brownian motion, but that it is actually singular with respect to the standard Wiener measure on $C([0,1])$.
\end{abstract}

\maketitle

\section{Introduction}

Polymers are sequences of molecules bonded together to form extremely long chains of chemical compounds, and their ubiquity, diversity, and crucial importance in chemistry, biology, physics, etc., has led to numerous attempts at building models of their formation and subsequent behavior. The tools of statistical physics and probability are well-suited for constructing and analyzing polymer models, and the literature of both fields abounds with many examples. Particularly nice surveys of the scope and complexity of polymer models in the probabilistic context are found in \cite{giacomin:book} and \cite{ComSY:review}. In this paper we concentrate on the (large) subclass of \textit{directed} models featuring interaction with external environments, as originally introduced in \cite{HuseHenley, ImSpen:diffusion}. The specifics of models of this type are varied, but from an abstract, purely mathematical point of view they are of a very similar flavor. Generally speaking, models for a discrete directed polymer interacting with an environment can be thought of as probability measures on the product space of environments and nearest neighbour paths, which is the usual approach taken for models of random walk in random environment. The measure is determined by specifying:
\begin{itemize}
\item the marginal measure of the environment, and,
\item the conditional measure on the path space, given the environment, as a Gibbs measure on the set of simple random walk paths of length $n$.
\end{itemize}
The Hamiltonian for the Gibbs measure is a function of the environment and (usually) a temperature parameter. Letting $\omega$ denote an environment and $S$ denote a generic path, the polymer measure at inverse temperature $\beta$ is written as
\begin{align*}
\P_{n, \beta} \left( d \omega d S \right) = \frac{1}{Z_n^{\omega}(\beta)} e^{\beta H_n^{\omega}(S)} \, \Pl(dS) Q(d \omega) = \Pl_{n, \beta}^{\omega}(dS) Q(d \omega),
\end{align*}
where $Q$ is the marginal measure of the environment, $\Pl$ is a reference measure on the set of paths (usually uniform measure), and $Z_n^{\omega}(\beta)$ is a partition function that normalizes the quenched measure $\Pl_{n, \beta}^{\omega}$ to have total mass one. In this framework, the emphasis is on understanding the behavior of the polymer by studying properties of the probability measure $\P_{n, \beta}$. Most commonly one is interested in understanding the measure as the length of the polymer grows large. In most previous work this has been done by studying various statistics of the polymer measure, such as the asymptotic free energy or the typical displacement of the polymer path from its starting point, but little has been done in understanding the asymptotic structure of the measure itself.

The purpose of the present paper is to understand the large $n$ limit of a directed polymer model in dimension $1$+$1$ by constructing the scaling limit of the corresponding probability measure on environments and paths. The measure we construct is truly a continuum object, meaning that it is supported on continuum environments and continuous paths. For this reason we call the limiting object the \textbf{continuum directed random polymer}. Its most novel feature is that, like in the discrete case, the continuum path measure actually depends on the continuum environment, and in a very non-trivial way. In fact we shall see that for almost all environments the continuum random polymer is singular with respect to Brownian motion, and hence in this regard it can be seen as a genuinely new measure on continuous paths. This is to be contrasted with the situation in \cite{ImSpen:diffusion, bolthausen:note, kifer:polymers, ComY:diffusive_weak}, where for $d \geq 3$ and $\beta$ sufficiently small the limiting measure is simply Brownian motion, independent of the environment, or even \cite{conlon_olsen}, where the limiting measure is Brownian motion with a diffusion constant that depends on the environment.

We emphasize that in the current paper we will not actually be taking any scaling limits of discrete polymers, rather we will construct the probability measure on continuum environments and paths directly from other continuum objects. In the related work \cite{AKQ:IDR} we show that the random field of transition probabilities for our path measure does arise as a scaling limit from $1+1$ dimensional directed polymers in the \textbf{intermediate disorder regime} \cite{AKQ:prl}. Under this regime the intensity of the environment is scaled to zero with the length of the polymer, and the continuum random polymer arises from a diffusive scaling of the discrete polymer paths. For the present paper, however, we do not use any results from the discrete model and we only refer to it for motivational purposes.

We also emphasize that there are many references to a continuum directed random polymer in the literature, see for example \cite{BC, BG, ACQ}, but to the best of our knowledge the specific one we construct has not been considered. In several papers \cite{BC, BTV08} polymers were studied by smoothing out the environment field in the space variable so that the random path measure can be constructed using a Radon-Nikodym derivative with respect to Wiener measure. Part of the purpose of this paper is to show that this smoothing of the white noise is not necessary as long as one does not try to directly construct a Gibbs measure (see \eqref{formal_gibbs_expression}). Our construction is based on a different technique: we specify the random transition probabilities of the polymer and show that they are almost surely consistent. This uniquely specifies the path measure, and as we will show it is almost surely singular with respect to Wiener measure. This explains why a Gibbs measure cannot be constructed directly. However, our construction is simple enough that we are able to derive several path properties of the continuum polymer, which we do in Section 4.

The paper is organized as follows: in the next section we give a brief, heuristic introduction to the continuum directed random polymer and its construction. We also include a short background section on the discrete polymer model against which the reader may compare the continuum polymer. Section 3 reviews some standard results on white noise, stochastic integration, and the stochastic heat equation, which are used in Section 4 to finally construct the continuum polymer. Precise statements of our theorems can also be found in Section 4. In Section 5 we discuss some generalizations of the continuum random polymer which can be easily constructed using the methods of Section 4. We finish by deriving a formal expression for the stochastic differential equation that describes the infinitesimal dynamics of the continuum random polymer, and demonstrates the connection with the Kardar-Parisi-Zhang equation.

\section{Heuristics and Motivation}

\subsection{A Heuristic Description of the Continuum Random Polymer}

The continuum random polymer is constructed using the framework of the introduction, but with continuum objects replacing discrete ones whenever necessary. Our polymer is $1$+$1$-dimensional, hence the relevant continuum objects live in the space $[0,1] \times \R$, where the first coordinate is time and the second is space. The continuum environment is a very basic random field: a white noise $W$ on $[0,1] \times \R$ with correlations $$\expect{W(s , y) W(t , x)} = \delta(t-s) \delta(x-y).$$ Formally, the Hamiltonian for the continuum random polymer is the integral of the white noise field over a Brownian path, leading to the Gibbsian formalism
\begin{align}\label{formal_gibbs_expression}
\P_{\beta} \left( dW \, dB \right) = \frac{1}{\mZ_{\beta}^W} \exp \left \{ \beta \int_0^1 W(s, B_s) \, ds \right \} \Pl(dB) \, Q(dW).
\end{align}
Here $\Pl$ represents the usual Wiener measure on $C([0,1])$. We emphasize that this is a purely formal expression that utilizes some attractive and suggestive notation but is not rigorous. The integral of white noise over a Brownian path is a distribution and not a function, and therefore the exponential cannot be given a direct meaning.

As the last paragraph hints, in contrast to the discrete case we are not able to express the path measure for the continuum random polymer in Gibbsian form. As mentioned in the introduction, we will see that the path measure is singular with respect to all other path measures that we know of (most importantly to Wiener measure), which leaves us with no reference measure against which we can define a Gibbs form. However, there is more than one way of constructing a path measure, and instead of using the Gibbsian formalism we exploit the fact that path measures are uniquely determined by their finite dimensional distributions. To construct the random path measure it is sufficient to construct a family of finite dimensional distributions (\textit{fdd}s) that are almost surely consistent and which allow the measure to be almost surely supported on continuous paths (the almost sure statements here are in the environment). This is the strategy employed in Section 4. To understand where the \textit{fdd}s come from, we remark that the continuum random polymer, conditioned on the environment, turns out to be a Markov process. This is not surprising as the same is true for the discrete polymer, which we recall in the next section. The Markov property, however, implies that to determine the \textit{fdd}s it is enough to specify the transition probabilities
\begin{align*}
\Pl_{\beta}^W \left( X_t \in dx | X_s \in dy \right)
\end{align*}
for $0 \leq s < t \leq 1$. If we let $\mZ(s,y;t,x;\beta)$ represent the density of the transition probabilities with respect to Lebesgue measure, then \eqref{formal_gibbs_expression} and the Feynman-Kac representation suggest that this function should satisfy the stochastic heat equation
\begin{align*}
\partial_t \mZ = \tfrac{1}{2} \partial_x^2 \mZ+ \beta W \! \mZ, \quad \mZ(s,y;s,x;\beta) = \delta_0(x-y).
\end{align*}
In Section 3 we briefly review properties of solutions to the $1$-dimensional stochastic heat equation and show that it can indeed be used to construct a consistent family of \textit{fdd}s that support the path measure on $C([0,1])$. Before proceeding with that, we quickly recall the discrete model of directed polymers in dimension $1$+$1$ as another way of motivating the strategy outlined above.

\subsection{Directed Polymers in Dimension $1+1$}

For directed polymers in dimension $1+1$ the environment space is $\Omega = \left \{ \omega : \N \times \Z \to \R \right \}$, equipped with a product measure $Q$ under which the variables $\omega(i,x)$ are independent and identically distributed (i.i.d.), and the Hamiltonian is defined on nearest-neighbor random walk paths by
\begin{align*}
H_n^{\omega}(S) = \sum_{i=1}^n \omega(i, S_i).
\end{align*}
We assume that the environment variables are mean zero and variance one and have finite exponential moments. The corresponding Gibbs measure on the $2^n$ paths of length $n$ is given by
\begin{align*}
\Pl_{n, \beta}^{\omega}(S) = \frac{1}{Z_n^{\omega}(\beta)} e^{\beta H_n^{\omega}(S)} \Pl(S),
\end{align*}
where $\Pl(S)$ is the uniform probability measure on paths and $Z_n^{\omega}(\beta)$ is the partition function
\begin{align*}
Z_n^{\omega}(\beta) = \E_{\Pl} \left[ e^{\beta H_n^{\omega}(S)} \right].
\end{align*}
There are two important properties of the path measure that we wish to emphasize:
\begin{itemize}
\item for a fixed $n$ and conditional on the environment the path measure is Markovian, and,
\item the path measure is fully determined by its finite dimensional distributions, and these can be expressed in terms of the \textit{space-time point-to-point} partition functions.
\end{itemize}
Both properties are inherited from the Markovian nature of the underlying path measure $\Pl$. For any set $A$ of random walk paths, the definition of $\Pl_{n, \beta}^{\omega}$ implies that
\begin{align*}
\Pl_{n,\beta}^{\omega}(A) = \frac{1}{Z_n^{\omega}(\beta)} \E_{\Pl} \left[ e^{\beta H_n^{\omega}(S)} \indicate{S \in A} \right].
\end{align*}
Taking $A$ to be the event $\{ S_{\vi_1} = \vx_1, \ldots, S_{\vi_k} = \vx_k, S_n = x \}$ for $1 \leq \vi_1 < \ldots < \vi_k < n$, we get
\begin{align}
\Pl_{n,\beta}^{\omega}(A) &= \frac{1}{Z_n^{\omega}(\beta)} \E_{\Pl} \left[ \prod_{m=0}^{k} \exp \left \{ \beta \!\!\!\! \sum_{l=\vi_{m}+1}^{\vi_{m+1}} \!\! \omega(l, S_l) \right \} \indicate{S_{\vi_{m+1}} = \vx_{m+1}} \right] \notag \\
&= \frac{1}{Z_n^{\omega}(\beta)} \prod_{m=0}^{k} \E_{\Pl} \left[ \left. \exp \left \{ \beta \!\!\!\! \sum_{l=\vi_{m}+1}^{\vi_{m+1}} \!\! \omega(l, S_l) \right \} \indicate{S_{\vi_{m+1}} = \vx_{m+1}} \right| S_{\vi_m} = \vx_m \right] \label{discrete_fdds_long}
\end{align}
Here $\vi_0 = \vx_0 = 0$ and $\vi_{k+1} = n, \vx_{k+1} = x$. The interchange of the product with the expectation is the expression of the fact that the random walk path passing through the sequence of space-time points $(\vi_j, \vx_j)$ is the concatenation of several independent walks (of appropriate length) starting from each of the points. We define the space-time \textit{point-to-point} partition functions by
\begin{align*}
Z_n^{\omega}(i,x;j,y; \beta) = \E_{\Pl} \left[ \left. \exp \left \{ \beta \! \sum_{k=i+1}^j \omega(k, S_k)  \right \} \indicate{S_j = y} \right| S_{i} = x \right],
\end{align*}
and the corresponding \textit{point-to-line} versions by
\begin{align*}
Z_n^{\omega}(i,x;j,*; \beta) = \sum_{y \in \Z} Z_n^{\omega}(i,x;j,y;\beta).
\end{align*}
Then the finite dimensional distributions \eqref{discrete_fdds_long} can be rewritten as
\begin{align*}
\Pl_{n, \beta}^{\omega} \left( S_{\vi_1} = \vx_1, \ldots, S_{\vi_k} = \vx_k, S_n = x \right) = \frac{1}{Z_n^{\omega}(0,0;n,*;\beta)} \prod_{m=0}^k Z_n^{\omega}(\vi_m, \vx_m; \vi_{m+1}, \vx_{m+1}; \beta).
\end{align*}
Summing over $x$ leads to
\begin{align}\label{discrete_fdds_free_endpoint}
\Pl_{n, \beta}^{\omega} \left( S_{\vi_1} = \vx_1, \ldots, S_{\vi_k} = \vx_k \right) = \frac{Z_n^{\omega}(\vi_k, \vx_k; n, *;\beta)}{Z_n^{\omega}(0,0;n,*;\beta)} \prod_{m=0}^{k-1} Z_n^{\omega}(\vi_m, \vx_m; \vi_{m+1}, \vx_{m+1}; \beta).
\end{align}
In particular this implies that for any integer $0 \leq i < n$ we have
\begin{align}
\Pl_{n, \beta}^{\omega} \left( \left. S_{i+1} = S_i \pm 1 \right| S_1, \ldots, S_i \right) &= Z_{n}^{\omega}(i, S_i; i+1, S_i \pm 1; \beta) \frac{Z_n^{\omega}(i+1, S_i \pm 1; n, *; \beta)}{Z_n^{\omega}(i, S_i; n, *; \beta)} \notag \\
&= \tfrac12 e^{\beta \omega(i+1, S_i \pm 1)} \frac{Z_n^{\omega}(i+1, S_i \pm 1; n, *; \beta)}{Z_n^{\omega}(i, S_i; n, *; \beta)} \label{discrete_one_step_trans_prob}
\end{align}
Given $\omega$ the right hand side is completely determined by $S_i$, hence under $\Pl_{n, \beta}^{\omega}$ the walk $S_i, 1 \leq i \leq n$, is a Markov process with \eqref{discrete_one_step_trans_prob} as its transition probabilities. As is to be expected they are \textit{not} homogeneous in space and time but instead vary according to the random environment. Note also that they depend on the \textit{entire} environment ahead of the current polymer position $(i,S_i)$, and not just the environment locally surrounding the space-time point. Equation \eqref{discrete_fdds_free_endpoint} gives the family of finite dimensional distributions for the polymer, and as the polymer can only be one of finitely many paths these f.d.d.s uniquely specify the path measure.

Observe that equation \eqref{discrete_one_step_trans_prob} is also a manifestation of the simple identity
\begin{align}\label{discrete_backward_difference_equation}
Z_n^{\omega}(i,x;n,*;\beta) = \tfrac12 e^{\beta \omega(i+1, x+1)} Z_n^{\omega}(i+1, x+1; n, *; \beta) + \tfrac12 e^{\beta \omega(i+1, x-1)} Z_n^{\omega}(i+1, x-1; n, *; \beta),
\end{align}
which is obtained by conditioning on the direction of the walk between time $i$ and $i+1$. This emphasizes that the point-to-point partition function can be fully determined given $\omega$; one simply has to solve the random difference equation \eqref{discrete_backward_difference_equation}. In probabilistic language it is a random backwards Kolmogorov equation with the final time data given by
\begin{align*}
Z_n^{\omega}(n,x;n,*;\beta) = e^{\beta \omega(n,x)}.
\end{align*}
There is also a corresponding forward evolution equation for the point-to-point field. By conditioning on the step between time $j$ and $j+1$ we get
\begin{align}\label{discrete_forward_difference_equation}
Z_n^{\omega}(i,x;j+1,y;\beta) = \tfrac{1}{2} e^{\beta \omega(j+1, y)} \left[ Z_n^{\omega}(i,x;j,y+1;\beta) + Z_n^{\omega}(i,x; j, y-1;\beta) \right]
\end{align}
with the initial data $Z_n^{\omega}(i,x;i,y) = \indicate{x=y}$. Both \eqref{discrete_backward_difference_equation} and \eqref{discrete_forward_difference_equation} can be considered as discrete time heat equations, although one is forward in time and the other is backwards in time.

We end this section with a brief remark that we consider important, even though it will not be used in the rest of the paper. In contrast to simple random walk it is important to recognize that the \textit{family} of measures $\Pl_{n, \beta}^{\omega}, n \geq 0$, is \textit{not} consistent, i.e. the measure on paths of length $n+1$ does not project onto the measure on paths of length $n$ when one simply forgets the time $n+1$ position of the walk. Intuitively this is clear. The polymer measure weights each path according to the energy assigned to it by the Hamiltonian, and the energy of a path after the first $n$ steps may vary wildly from the energy after $n+1$ steps, depending on the values of $\omega$ at time $n+1$. Hence the probability of the length $n$ path may be very different under the two measures. This can also be seen from equation \eqref{discrete_one_step_trans_prob} by noticing that at a fixed time $i$ the transition probabilities are not constant in $n$. This lack of consistency is also a feature of the continuum random polymer.

\section{The Continuum Stochastic Heat Equation}

\subsection{White Noise on $[0,1] \times \R$}

\newcommand{\Qf}{\mathcal{W}}

In this section we briefly recall the elementary theory of white noise and stochastic integration on the particular measure space $L^2([0,1] \times \R, \B, dt \, dx)$. Here $\B$ is the $\sigma$-algebra of Borel subsets and $dt \, dx$ denotes Lebesgue measure on the space. We let $\B_f$ be the subset of $\B$ consisting of sets of finite Lebesgue measure. Observe that $\B = \sigma(\B_f)$ because Lebesgue measure is $\sigma$-finite on the given space.

A white noise on $[0,1] \times \R$ is a collection of mean zero Gaussian random variables defined on a common probability space $(E, \Qf, Q)$ and indexed by $\B_f$:
\begin{align*}
W = \left \{ W(A) : A \in \B_f \right \}.
\end{align*}
This means that every finite collection of the form $(W(A_1), \ldots, W(A_k))$ has a $k$-dimensional Gaussian distribution, with mean zero and covariance structure given by
\begin{align*}
\expect{W(A)W(B)} = \left| A \cap B \right|.
\end{align*}
In particular if $A$ and $B$ are disjoint then $W(A)$ and $W(B)$ are independent.

Stochastic integration uses the white noise to map elements of $L^2([0,1] \times \R)$ into Gaussian random variables in a way that preserves the Hilbert space structure of the $L^2$ space. For $g \in L^2([0,1] \times \R)$ we use the notation
\begin{align*}
I_1(g) = \int_0^1 \int g(t,x) \, W(t , x) \, dt \, dx,
\end{align*}
with the integral being defined as an appropriate $L^2(E, \Qf, Q)$ limit, see \cite{Janson:GHS} for a more detailed definition. Multiple stochastic integrals $I_k$ map elements of $g \in L^2([0,1]^k \times \R^k)$ into $L^2(E, \Qf, Q)$ random variables, for which we use the notation
\begin{align*}
I_k(g) = \int_{[0,1]^k} \int_{\R^k} g(\vt, \vx) \, W^{\otimes k}(\vt, \vx) \, d \vt \, d \vx.
\end{align*}
These integrals are often used to express solutions to stochastic partial differential equations. The most important fact we will use about them is the covariance relation
\begin{align}\label{stoch_integral_covariance}
\E_{Q} \left[ I_k(g) I_j(h) \right] = \indicate{j = k} \left< g, h \right>_{L^2([0,1]^k \times \R^k)}.
\end{align}
This implies a scaling relation for $W$. Since $W(t,x)$ can be thought of as mapping elements of $L^2([0,1] \times \R)$ into Gaussians with variance given by the square of the $L^2$ norm, formally applying a change of variables shows that for $\alpha, \beta > 0$ the noise $W(\alpha t, \beta x)$ maps $g \in L^2([0,1] \times \R)$ into a Gaussian with variance $\alpha ^{-1} \beta^{-1} ||g||_{L^2([0,1] \times \R)}^2$. Hence this gives us the scaling relation
\begin{align*}
W(\alpha t, \beta x) = \alpha^{-1/2} \beta^{-1/2} W(t,x).
\end{align*}
More generally, for $k \geq 1$ we have
\begin{align}\label{white_noise_scaling}
W^{\otimes k}(\alpha \vt, \beta \vx) = \alpha^{-k/2} \beta^{-k/2} W^{\otimes k}(\vt, \vx).
\end{align}

\subsection{Solutions and Properties of the Stochastic Heat Equation}

The stochastic heat equation with multiplicative noise plays a central role in our construction of the continuum directed random polymer. We use it to construct the four-parameter field which defines the continuum polymer's finite dimensional distributions. The basic equation is
\begin{align}\label{SHE}
\partial_t \mZ = \tfrac{1}{2} \partial^2_{x} \mZ + \beta W \! \mZ.
\end{align}
We will only consider solutions with delta function initial conditions,
\begin{align}\label{SHEinit}
\lim_{t\downarrow 0} \mZ(t,x)= \delta_0(x),
\end{align}
where \eqref{SHEinit} is interpreted in the weak sense, meaning that the limit holds almost surely when $\mZ$ is integrated against smooth functions of compact support. The presence of the white noise means that any reasonable solution of \eqref{SHE} will not be differentiable in space or time. Instead of requiring differentiability we say that $\mZ = \mZ(t,x)$ is a {\it mild solution} to \eqref{SHE} if it satisfies the integral equation
\begin{align}\label{SHE_integral_equation}
\mZ(t,x) = \int \varrho(t-s,x-y)\mZ(s,y) \, dy + \beta \int_s^t \int \varrho(t-u,x-y) \mZ(u,y) W(u , y) \, du \, dy,
\end{align}
for all $0 \leq s < t$, and if for every $(t,x) \in [0,1] \times \R$ the random variable $\mZ(t,x)$ is $\Qf_t := \sigma(W(s, \cdot) : 0 \leq s \leq t)$ measurable. This measurability condition makes the stochastic integral in the right hand side of \eqref{SHE_integral_equation} well defined in the It\^o sense. Here $\varrho$ the usual heat kernel
\begin{align*}
\varrho(t,x) = \frac{e^{-x^2/2t}}{\sqrt{2 \pi t}}.
\end{align*}

We also have an explicit Wiener chaos expansion of $\mZ(t,x)$.  Since we only use it in the case
of delta initial conditions, we assume here that $\mZ(0,x) =\delta_0(x)$. Iterating \eqref{SHE_integral_equation} multiple times gives
\begin{align}
\mZ(t,x) = \varrho(t,x) &  + \beta \int_0^t \int \varrho(t - s, x - y) \varrho(s,y) W(s , y) \, ds \, dy\label{SHE_chaos_expansion} \\
& \!\!\!\!\!\!\!\!\!\!\!\!\! + \beta^2 \int_0^t \int_0^s \int \int \varrho(t-s,x-y) \varrho(s-r, y-z) \varrho(r,z) W(s , y) W(r , z) \, ds \, dr \, dy \, dz + \ldots \notag
\end{align}
The stochastic integrals $I_k$ of the last section allow us to write this in more succinct notation, and to slightly generalize. For each $(s,y), (t,x) \in [0,1] \times \R$ with $0 \leq s < t \leq 1$, define the Brownian motion transition probabilities via
\begin{align*}
\varrho(s,y;t,x) = \varrho(t-s, x-y),
\end{align*}
and for $(\vt, \vx) \in [0,1]^k \times \R^k$ define the multi-step transition probabilities by
\begin{align*}
\varrho_k(\vt, \vx | s,y;t,x) = \varrho(s,y; \vt_1, \vx_1) \left( \prod_{j=1}^{k-1} \varrho(\vt_j, \vx_j; \vt_{j+1}, \vx_{j+1}) \right) \varrho(\vt_k, \vx_k, t,x) \indicate{\vt \in \Delta_k(s,t]}.
\end{align*}
Here $\Delta_k(s, t] = \{ \vt \in \R_+^k : s < \vt_1 < \vt_2 < \ldots < \vt_k \leq t \}.$  For each $\beta \geq 0$, we \textbf{define} the four-parameter random field $\mZ(s,y;t,x;\beta)$ by
\begin{align}\label{four_param_field_def}
\mZ(s,y;t,x;\beta) = \sum_{k=0}^{\infty} \beta^k I_k(\varrho_k(\cdot, \cdot | s,y;t,x)).
\end{align}
Here $I_0(\varrho_0(\cdot, \cdot | s,y;t,x)) = \varrho(s,y;t,x) = \varrho(t-s,x-y)$, and in this new notation equation \eqref{SHE_chaos_expansion} is simply $\mZ(0,0;t,x;\beta)$. A comparison with \eqref{SHE} and \eqref{SHE_chaos_expansion} shows that for each $(s,y)$ the field $(t,x) \mapsto \mZ(s,y;t,x;\beta)$ is the solution to
\begin{align}\label{generalized_SHE}
\partial_t \mZ = \tfrac{1}{2} \partial_{x}^2 \mZ + \beta W Z, \quad \lim_{t \downarrow s} \mZ(s,y;t,x) = \delta_0(x-y).
\end{align}
It is easy to see that $\mZ(s,y;t,x;\beta)$ satisfies the integral equation
\begin{align}\label{generalized_integral_equation}
\mZ(s,y;t,x;\beta) &= \int \varrho(t-r,x-z) \mZ(s,y;r,z;\beta) \, dz \notag \\
&\quad \, + \beta \int_r^t \int \varrho(t-u, x-z) \mZ(s,y;u,z;\beta) W(u,z) \, du \, dz
\end{align}
and using the covariance relation \eqref{stoch_integral_covariance} it follows that
\begin{align}\label{Z_covariance_relation}
\E_{Q} \left[ \mZ(s,y;t,x;\beta)^2 \right] = \sum_{k=0}^{\infty} \beta^{2k} || \varrho_k(\cdot, \cdot | s,y;t,x) ||_{L^2([0,1]^k \times \R^k)}^2.
\end{align}
Using integrals that appear in the Dirichlet distribution, it is straightforward to compute that
\begin{align}\label{rhok_norm}
||\varrho_k(\cdot, \cdot | s,y;t,x)||_{L^2([0,1]^k \times \R^k)}^2 = \frac{(t-s)^{\frac{k}{2} - 1}}{2^{k+1} \sqrt{\pi} } \frac{e^{-(x-y)^2/(t-s)}}{\Gamma(\frac{k+1}{2})},
\end{align}
which is clearly summable in $k$.  Hence one has that (\ref{four_param_field_def}) is a mild
solution of (\ref{SHE}). From \eqref{Z_covariance_relation} and \eqref{rhok_norm} it follows that for all $\beta \geq 0$ there exists a constant $C_\beta<\infty$ such that
\begin{equation}\label{elltwobd}
E_{Q}\left[ \mZ(s,y;t,x;\beta)^2 \right] \le C_\beta \varrho(s,y;t,x)^2
\end{equation}
for all $0 < s \leq t \leq 1$. In \cite{BC} it is shown that there exists at most one mild solution to \eqref{SHE} with initial condition \eqref{SHEinit} satisfying
\begin{equation}\label{elltwobd2}
\sup_{\substack{0\le t\le 1 \\ x\in\mathbb{R}}} \int_{0 < s' < s \le t} ds \, ds' \int_{\mathbb{R}^2} dy \, dy'
E_Q[ \mZ(s',y')^2]\varrho(s',y',s,y)^2 \varrho(s,y,t,x) <\infty.
\end{equation}
Precisely, this means that any two fields $\mZ(t,x), \mZ^*(t,x)$ satisfying \eqref{SHE_integral_equation} (for the same white noise) are equal almost surely, in the sense that
\begin{align*}
Q \left( \mZ(t,x) = \mZ^*(t,x) \textrm{ for all } 0 \leq  t \leq 1, x \in \R \right) = 1.
\end{align*}
Equation \eqref{elltwobd2} follows from \eqref{elltwobd} for delta function initial data, hence we have that \eqref{SHE_chaos_expansion} is the unique mild solution to \eqref{SHE} with initial condition \eqref{SHEinit}. Since the equation \eqref{generalized_SHE} differs from \eqref{SHE} only by the location of the initial condition, the method of \cite{BC} extends to show that \eqref{four_param_field_def} is the unique mild solution to \eqref{generalized_SHE}. Hence any two fields satisfying \eqref{generalized_integral_equation} are equal almost surely.

\begin{remark}
In the context of stochastic partial differential equations, mild solutions are the analogue of {\it strong} solutions of stochastic differential equations in the sense that they are solutions of a stochastic equation involving a given white noise. Another concept of solution is the {\it solution of the martingale problem}, which is to find a probability measure $\mu$ on $C(\mathbb{R}_+,C(\mathbb{R}))$ (the distribution of the $\mZ(t,x)$) under which both
\begin{align*}
M_t(\varphi):=  \int \mZ(t,x)\varphi(x) \, dx - \int  \mZ(t,x)\varphi(x) \, dx - \int_0^t \int  \mZ(s,x) \varphi''(x) \, dx \, ds
\end{align*}
and
\begin{align*}
\Lambda_t(\varphi) := M_t(\varphi)^2 - \int_0^t \int  \mZ(s,x)^2\varphi(x)^2 \, dx \, ds
\end{align*}
are local martingales with respect to the filtration $\Qf_t$ for every smooth function $\varphi$ of compact support in $\mathbb{R}$. It is shown in \cite{BG} that given a solution of the martingale problem there exists a probability space on which there is a white noise and  $\mZ(t,x)$ distributed according to $\mu$ such that $\mZ$ is a mild solution of (\ref{SHE}).
\end{remark}

\begin{remark}
There is a convenient shorthand expression for our four parameter random field: it may be rewritten as
\begin{align}\label{exponential_shorthand}
\mZ(s,y;t,x;\beta) = \E_{\Pl} \left[ \left. \wp{\exp} \left \{ \beta\! \int_s^t \! W(u, B_u) \, du \right \} \right| B_s \in dy, B_t \in dx \right] \varrho(s,y;t,x)
\end{align}
Formally the expression inside the expectation makes no sense, since we do not know how to interpret the integral of a white noise over a Brownian path. However, if one takes the exponential and expands it into a Taylor series, then switches the expectation over paths with the infinite sum and integrates out the Brownian terms $B$, the well-defined chaos expansion \eqref{four_param_field_def} appears. The Wick exponential $\wp{\exp}$ indicates that powers of integrals should be expanded via the rule
\begin{align*}
: \! \left( \int_0^1 W(s, B_s) \, ds \right)^k \!\!\! : \,\,= k! \! \int_{\Delta_k} \prod_{j=1}^k W(t_j, B_{t_j}) \, dt_j.
\end{align*}
Equation \eqref{exponential_shorthand} should be seen as a convenient shorthand for this procedure.
\end{remark}

The theorem below summarizes the properties of the field $\mZ$ that we will use in later sections.

\begin{theorem}\label{Z_properties_theorem}
There exists a version of the field $\mZ(s,y;t,x;\beta)$ which is jointly continuous in all four variables and has the following properties:
\begin{enumerate}[i.]
\item $\E_{Q} \left[ \mZ(s, y; t, x; \beta) \right] = \varrho(s,y;t,x)$;
\item \textbf{stationarity:} $\mZ(s, y; t, x; \beta) \equiv \mZ(s + u_0, x + z_0; t + u_0, y + z_0; \beta)$;
\item \textbf{scaling:} $\mZ(r^2 s, ry; r^2 t, rx; \beta) \equiv r^{-1} \mZ(s, y; t, x; \beta \sqrt{r})$;
\item \textbf{positivity:} \cite{Mueller} with $Q$ probability one, $\mZ(s,y;t,x;\beta)$ is strictly positive for all tuples $(s,y;t,x)$ with $0 \leq s < t \leq 1$,
\item the law of $\mZ(s,y;t,x;\beta)/\varrho(s,y;t,x)$ does not depend on $x$ or $y$,
\item it has an independence property among disjoint time intervals: for any finite disjoint $\{ (s_i, t_i] \}_{i=1}^n$ and any $x_i, y_i \in \R$, the random variables $\{ \mZ(s_i, y_i; t_i, x_i; \beta) \}_{i=1}^n$ are mutually independent,
\item \textbf{Chapman-Kolmogorov equations}: with $Q$ probability one, for all $0 \leq s < r < t \leq 1$ and $x,y \in \R$,
\begin{align*}
\mZ(s,y;t,x; \beta) = \int \! \mZ(s,y;r,z; \beta) \mZ(r,z;t,x;\beta) \, dz.
\end{align*}
\end{enumerate}
\end{theorem}

\begin{proof}
The continuity of the field in $t$ and $x$ is proven in \cite{walsh:st_flour}, and the method is easily extended to show the continuity in $s$ and $y$. Property (i) follows from the $k \geq 1$ terms of \eqref{four_param_field_def} having mean zero. Property (ii) is a consequence of the translation invariance of the white noise, and property (iii) follows from the scaling relation for white noise and the Brownian scaling property $\varrho(r^2 s, ry; r^2 t, rx) = r^{-1} \varrho(s,y;t,x)$.

Property (iv) is proved in \cite{Mueller} for solutions to \eqref{SHE} with bounded, non-negative functions with compact support (that are not identically zero) as the initial conditions. The extension of his result to delta function initial conditions is proved in \cite{BC}.

For property (v) first note that by the scaling and translation properties it is enough to consider $s = y = 0$ and $t = 1$, which only leaves to show that the law does not depend on $x$. This is a result of the translation properties of Brownian bridges: a bridge from zero to a given endpoint is simply the translate of a bridge from zero to zero along a straight line. Thus
\begin{align*}
\E_{BB} \left[ \wp{\exp} \left \{ \beta\! \int_s^t \! W(u, X_u + u(x+\theta)) \, du \right \} \right] = \E_{BB} \left[ \wp{\exp} \left \{ \beta\! \int_s^t \! W_\theta(u, X_u + ux) \, du \right \} \right],
\end{align*}
where in this instance $X$ is a Brownian bridge starting and ending at zero, and $W_\theta(t,x) = W(t, x + \theta t)$. Since the map $(t,x) \mapsto (t, x + \theta t)$ preserves area the field $W_{\theta}$ is a white noise with correlations the same as $W$, which proves the invariance.

Property (vi) follows from the fact that $\mZ$ on each of the time intervals depends only on the white noise in that interval, and each of these white noises are independent of each other.

Property (vii) follows from linearity and uniqueness of solutions to the stochastic heat equation. Consider the field
\begin{align}\label{intermediate_field}
\int \mZ(s,y;r,z;\beta) \mZ(r,z;t,x;\beta) \, dz.
\end{align}
Linearity of the stochastic heat equation implies that this field satisfies \eqref{generalized_integral_equation} at any fixed, and hence at all rational, $(s,y;t,x)$ and $r$. Since \eqref{intermediate_field} is a continuous function in the four variables it therefore satisfies the integral equation \eqref{generalized_integral_equation} at all $(s,y;t,x)$ and $r$. But this means that it is a mild solution to \eqref{generalized_SHE}, with the same white noise that is used to construct the solution $\mZ(s,y;t,x;\beta)$, and therefore the uniqueness implies that \eqref{intermediate_field} and $\mZ(s,y;t,x;\beta)$ are equal almost surely.
\end{proof}

In polymer language the variables $\mZ(s,y;t,x;\beta)$ are \textit{point-to-point} partition functions. We will also need to make use of \textit{point-to-line} partition functions, which are integrals of the point-to-point versions. Analogously to the notation for discrete polymers, we define
\begin{align}\label{continuum_p2l_defn}
\mZ(s,y;t,*;\beta) = \int \mZ(s,y;t,x; \beta) \, dx.
\end{align}
In the convenient shorthand this can also be written as
\begin{align}\label{continuum_p2l_shorthand}
\mZ(s,y;t,*;\beta) = \Pl \left[ \left. \wp{\exp} \left \{ \beta \! \int_s^t \! W(u, B_u) \, du \right \} \right| B_s \in dy \right].
\end{align}
It is easy to see that $\E_Q \left[ \mZ(s,y;t,*;\beta) \right] = 1$, and that the positivity, translation invariance, scaling relation, and independence properties of the point-to-point partition functions all carry over to these point-to-line versions as well.

\section{Construction and Properties of the Continuum Directed Random Polymer}

The purpose of this section is to construct, for each $\beta \geq 0$, a probability measure $\P_{\beta}$ on the product space of continuum environments and continuum paths. We let $\F$ denote the $\sigma$-algebra on the product space. Throughout the environment will be white noise on $[0,1] \times \R$, hence we will focus mostly on constructing the path measure $\Pl_{\beta}^W$ for each realization of $W$.

\subsection{Finite Dimensional Distributions}

In this section we construct the probability measure $\Pl_{\beta}^W$ on $C([0,1])$ that defines the continuum random polymer. Let $X_t : C([0,1]) \to \R$, $0 \leq t \leq 1$, be the standard coordinate mappings given by
\begin{align*}
X_t(x) = x(t).
\end{align*}
We equip $C([0,1])$ with the standard topology of the cylinder sets, that is the smallest one that makes the mappings $X_t$ measurable. The natural filtration is $\mathcal{G}_t := \sigma \left( X_s : 0 \leq s \leq t \right)$. Recall that a probability measure on this space is uniquely determined by its finite dimensional distributions. We exploit this fact to define the continuum random polymer.

\begin{definition}\label{continuum_fdds}
Conditional on the white noise $W$, let $\Pl_{\beta}^W$ be the measure on $C([0,1])$ whose finite dimensional distributions are given by
\begin{align}\label{continuum_fdds_eqn}
\Pl_{\beta}^W \! \left( X_{\vt_1} \in d\vx_1, \ldots, X_{\vt_k} \in d\vx_k \right) = \frac{1}{\mZ(0,0;1,*;\beta)} \prod_{j=0}^{k} \mZ(\vt_j, \vx_j; \vt_{j+1}, \vt_{j+1};\beta) \, d\vx_1 \ldots d\vx_k,
\end{align}
with $(\vt_0, \vx_0) = (0,0)$ and $(\vt_{k+1}, \vx_{k+1}) = (1,*)$. The joint measure $\P_{\beta}$ is then defined by
\begin{align*}
\P_{\beta}(dW dX) = Q(dW) \Pl_{\beta}^W(dX).
\end{align*}
The $\beta = 0$ case we denote by $\P$.
\end{definition}

Observe that under $\P$ the path measure is standard Wiener measure and that the path and the environment are independent of each other. For all $\beta \geq 0$ the Chapman-Kolmogorov equations for the field $\mZ$ imply that these finite dimensional distributions are consistent, and that $\Pl_{\beta}^W$ is a probability measure. This definition also implies the following result, which we will use repeatedly.

\begin{lemma}\label{continuum_fdd_average}
For each $\beta \geq 0$ we have the formula
\begin{align*}
\E_{Q} \left[ \mZ(0,0;1,*;\beta) \Pl_{\beta}^W \! \left( X_{\vt_1} \in d\vx_1, \ldots, X_{\vt_k} \in d\vx_k \right) \right] = \Pl \! \left( X_{\vt_1} \in d\vx_1, \ldots, X_{\vt_k} \in d\vx_k \right).
\end{align*}
More generally, if $Y : C[0,1] \to \R$ is integrable with respect to $\Pl$ then
\begin{align*}
\E_{\P_{\beta}} \left[ \mZ(0,0;1,*; \beta) Y \right] = \E_{Q} \left[ \mZ(0,0;1,*;\beta) \E_{\Pl_{\beta}^W}[Y] \right] = \E_{\Pl}[Y] = \E_{\P}[Y].
\end{align*}
\end{lemma}

\begin{proof}
The first formula follows from the properties of $\mZ$ outlined in Theorem \ref{Z_properties_theorem}. After multiplying by $\mZ(0,0;1,*)$, the right hand side of \eqref{continuum_fdds_eqn} is the product of independent random variables, each of which has mean $\varrho(\vt_j, \vx_j; \vt_{j+1}, \vx_{j+1})$, and hence their product is the occupation density for Brownian motion. The second general formula now follows from the first one by approximation of $Y$ with simple functions.
\end{proof}

\subsection{H\"{o}lder Continuity}

\begin{theorem}
With $Q$ probability one, the probability measure $\Pl_{\beta}^W$ is supported on $C([0,1])$ and the paths are H\"{o}lder continuous with exponent $\delta$, for every $\delta < 1/2$.
\end{theorem}

\begin{proof}
The proof is an application of the Garsia, Rodemich and Rumsey inequality, see \cite{varadhan:stoch_processes_notes} for an expository introduction.
For each $\gamma > 0$, define the random variable $Y_{\gamma}$ on $C([0,1])$ by
\begin{align*}
Y_{\gamma} = \int_0^1 \int_0^1 \frac{|X_t - X_s|^{2 \gamma}}{|t-s|^{\gamma}} \, dt \, ds.
\end{align*}
Garsia, Rodemich, and Rumsey says that, on a path-by-path basis,
\begin{align*}
\sup_{\substack{0 \leq s, t \leq 1 \\ |t-s| \leq \delta}} |X_t - X_s| \leq 8 \left( 4 Y_{\gamma} \right)^{\frac{1}{2 \gamma}} \delta^{\tfrac12 - \frac{1}{\gamma}}.
\end{align*}
Hence it is sufficient to show that, for all $\gamma > 0$, $Y_{\gamma} < \infty$ with $\P_{\beta}$ probability one. However, by Lemma \ref{continuum_fdd_average} we know that
\begin{align*}
\E_{\P_{\beta}} \left[ \mZ(0,0;1,*; \beta) Y_{\gamma} \right] = \E_Q \left[ \mZ(0,0;1,*; \beta) E_{\Pl^W_{\beta}}[Y_{\gamma}] \right] = \E_{\Pl}[Y_{\gamma}],
\end{align*}
and it is an easy computation for Brownian motion to show that the right hand side is finite. Since $\mZ(0,0;1,*)$ is strictly positive with $Q$ probability one (part iv of Theorem \ref{Z_properties_theorem}), it must be that $\E_{\Pl^W_{\beta}}[Y_{\gamma}]$ is finite with $Q$ probability one, or equivalently that $Y_{\gamma}$ is finite $\P_{\beta}$ almost surely.
\end{proof}

\subsection{Quadratic Variation}

We show that the quadratic variation of the continuum polymer behaves the same as Brownian motion if the spacing of the time mesh goes to zero fast enough.

\begin{theorem}\label{quadratic_variation_theorem}
Let $t_k^n = k 2^{-n}$. Then with $\P_{\beta}$ probability one, we have that for all $0 \leq t \leq 1$
\begin{align*}
\sum_{k=1}^{\lfloor 2^n t \rfloor} \left( X(t_k^n) - X(t_{k-1}^n) \right)^2 \to t
\end{align*}
as $n \to \infty$.
\end{theorem}

\begin{proof}
It is sufficient to consider $t = 1$ since the same argument carries over to all rational $t$ in $[0,1]$, and therefore the limit of the left hand side (as a function of $t$) exists pointwise at all rational times and is equal to the identity function. Since the limit is necessarily non-decreasing in $t$ it must be everywhere equal to the identity.

For $n \geq 1$ and $1 \leq k \leq 2^{n}$, let $I_{k,n} = (X(t_{k}^n) - X(t_{k-1}^n))^2 - 2^{-n}$, and
\begin{align*}
Y_n = \sum_{k=1}^{2^n} I_{k,n}.
\end{align*}
We will show that $Y_n^2 \to 0$ with $\P_{\beta}$ probability one. Using Lemma \ref{continuum_fdd_average} we have
\begin{align*}
\E_{\P_{\beta}} \left[ \mZ(0,0;1,*;\beta) I_{j,n} I_{k,n} \right] = \E_{\P} \left[ I_{j,n} I_{k,n} \right],
\end{align*}
and it is an easy computation for Brownian motion that $\E_{\P} \left[ I_{j,n} I_{k,n} \right] = \indicate{j=k} 2^{-2n + 1}$. Therefore
\begin{align*}
\E_{\P_{\beta}} \left[ \mZ(0,0;1,*;\beta) Y_n^2 \right] = \sum_{j,k=1}^{2^n} \E_{\P_{\beta}} \left[ \mZ(0,0;1,*;\beta) I_{j,n} I_{k,n} \right] = 2^{-n+1},
\end{align*}
so by Borel-Cantelli we have $\mZ(0,0;1,*;\beta) Y_n^2 \to 0$ with $\P_{\beta}$ probability one. But since $\mZ(0,0;1,*;\beta)$ is strictly positive this forces that $Y_n^2 \to 0$ with $\P_{\beta}$ probability one.
\end{proof}

\subsection{Singularity with Respect to Wiener Measure}

The last two sections show that the continuum random polymer has the same H\"{o}lder continuity and quadratic variation properties as Brownian motion. Based on this, it is natural to believe that the polymer is just a Brownian motion with a drift, and in some sense this is true but the drift is very rough. In Section \ref{sec:SDE} we show formal calculations that give a formula for the drift, even though we do not know how to properly make sense of it.

In this section we show that for all $\beta > 0$ the measure $\P_{\beta}$ is singular with respect to $\P$. Since the white noise environment is the same under both measures, this is equivalent to saying that for almost all realizations of the white noise the measure $\Pl_{\beta}^W$ is singular with respect to standard Wiener measure $\Pl$ on $C([0,1])$. In particular this shows that whatever the drift is it \textit{cannot} be in the Cameron-Martin class \cite{Janson:GHS}.

The idea of the proof is that the distribution (under $\P$ and $\P_{\beta}$) of the vectors
\begin{align*}
\left( X(0), X(2^{-n}), X(2 \cdot 2^{-n}), \ldots, X(1) \right)
\end{align*}
are absolutely continuous with respect to each other, but as $n \to \infty$ the Radon-Nikodym derivatives go to zero for $\P$ almost all paths. That this implies the measures are singular is a standard fact: see, for instance, \cite[Section 4.3.3]{durrett:book} for full details.

\newcommand{\ms}{\hspace{.75pt}}

As before, we let $t_k^n = k 2^{-n}$, and define the filtration by
\begin{align*}
\F_n = \sigma \left( W, X(t_k^n) : k \in \{ 0, 1, \ldots, 2^n \} \right).
\end{align*}
Observe that $\F_n \uparrow \F$ in the sense that $\sigma( \cup_n \F_n ) = \F$. For $x,y \in \R$ and $X \in C([0,1])$ let $C_k^n$ denote the following space-time tuples:
\begin{align*}
C_{k}^n(x,y) &= (t_{k-1}^n, x; t_k^n, y), \\
C_{k}^n(x, X) = C_k^n(x, X(t_k^n)), C_k^n(X, y) &= C_k^n(X(t_{k-1}^n), y), C_k^n(X) = C_k^n(X(t_{k-1}^n), X(t_{k}^n)).
\end{align*}
By Definition \ref{continuum_fdds}, it is easy to see that for $X \in C[0,1]$ we have
\begin{align*}
\left. \frac{d \ms \P_{\beta}}{d \ms \P} \right|_{\F_n} \!\!\!\!\!\!\! (W, X) = \left. \frac{d \! \Pl_{\beta}^W}{d \! \Pl} \right|_{\F_n} \!\!\!\!\!\!\! (X) = \frac{1}{\mZ(0,0; 1, *; \beta)} \prod_{k=1}^{2^n} \frac{\mZ(C_k^n(X); \beta)}{\varrho(C_k^n(X))}
\end{align*}
Let $M_n^W(X; \beta)$ denote the product term. We show the following:

\begin{theorem}\label{singularity_theorem}
For fixed $\beta \geq 0$, the sequence $M_n^W(X; \beta)$ is a positive martingale with respect to $\F_n$ and $\P$, and hence almost surely has a limit. If $\beta > 0$, then $\P$ almost surely $M_n^W(X; \beta) \to 0$ as $n \to \infty$.
\end{theorem}

\begin{proof}
First recall that we are working entirely under $\P$, and hence the white noise $W$ and the path $X$ are independent, and $X$ is a standard Brownian motion. That $M_n^W$ is a martingale follows from its definition as the Radon-Nikodym derivative on the increasing $\sigma$-fields $\F_n$, but we will also check it explicitly; in this calculation $\beta$ is fixed and so we drop it for convenience. Observe that
\begin{align*}
M_{n+1}^W(X) = \prod_{k=1}^{2^n} \frac{\mZ(C_{2k-1}^{n+1}(X)) \mZ(C_{2k}^{n+1}(X))}{\varrho(C_{2k-1}^{n+1}(X)) \varrho(C_{2k}^{n+1}(X))}.
\end{align*}
From the Markov property of standard Brownian motion it follows that
\begin{align}\label{continuum_martingale}
\E_{\P} \left[ \left. M_{n+1}^W(X) \right| \F_n \right] = \prod_{k=1}^{2^n} \E_{\P} \left[ \left. \frac{\mZ(C_{2k-1}^{n+1}(X)) \mZ(C_{2k}^{n+1}(X))}{\varrho(C_{2k-1}^{n+1}(X)) \varrho(C_{2k}^{n+1}(X))} \right| W, X(t_{k-1}^n), X(t_k^n) \right].
\end{align}
Each of these expectations is over Brownian motions conditioned to pass through $X(t_{k-1}^n)$ and $X(t_k^n)$, that is Brownian bridges. The term to integrate out is $X(t_{2k-1}^{n+1})$, which appears in the second space-time tuple of $C_{2k-1}^{n+1}(X)$ and the first space-time tuple of $C_{2k}^{n+1}(X)$. The conditional density of this term is
\begin{align}\label{continuum_singular_cancellation}
\P \left( \left. X(t_{2k-1}^{n+1}) \in dx \right| W, X(t_{k-1}^n), X(t_{k}^n) \right) &=
\Pl \left( \left. X(t_{2k-1}^{n+1}) \in dx \right| X(t_{k-1}^n), X(t_{k}^n) \right) \notag \\ &= \frac{\varrho(C_{2k-1}^{n+1}(X, x)) \varrho(C_{2k}^{n+1}(x, X))}{\varrho(C_{k}^n(X))} \, dx.
\end{align}
The numerator terms of \eqref{continuum_singular_cancellation} cancel with the denominator terms in \eqref{continuum_martingale}, and hence
\begin{align*}
\E_{\P} & \left[ \left. \frac{\mZ(C_{2k-1}^{n+1}(X)) \mZ(C_{2k}^{n+1}(X))}{\varrho(C_{2k-1}^{n+1}(X)) \varrho(C_{2k}^{n+1}(X))} \right| W, X(t_{k-1}^n), X(t_k^n) \right] \\
&= \frac{1}{\varrho(C_k^n(X))} \int \mZ \left( C_{2k-1}^{n+1}(X, x) \right) \mZ \left( C_{2k}^{n+1}(x, X) \right) \, dx = \frac{\mZ(C_k^n(X))}{\varrho(C_k^n(X))}
\end{align*}
by the Chapman-Kolmogorov equations. But this means that \eqref{continuum_martingale} is exactly $M_n^W(X)$. Hence $M_n^W(X)$ is a positive martingale, and therefore converges $\P$ almost surely.

Now we show that for $\beta > 0$ the limit is always zero. We proceed by examining the distribution of $M_n^W(X)$ as $n \to \infty$. First recall that for $X \in C([0,1])$ we have
\begin{align*}
\E_{Q} \left[ \mZ(C_k^n(X)) \right] = \varrho(C_k^n(X)),
\end{align*}
and that for $k \neq j$ the random variables $\mZ(C_k^n(X))$ and $\mZ(C_j^n(X))$ are independent of each other. Thus $M_n^W(X)$ is the product of $2^n$ independent and identically distributed random variables with mean $1$.

Now we analyze the distribution of the individual random variables. By the translation invariance of Brownian motion and its scaling relations, it follows that
\begin{align*}
C_k^n(X) \equiv (0, X(0); 2^{-n}, X(2^{-n})) \equiv (0, 0; 2^{-n}, 2^{-n/2} X(1)).
\end{align*}
Using that $X$ and $W$ are independent under $\P$, and combining the above with the scaling relation for $\mZ$ gives
\begin{align*}
\mZ(C_k^n(X); \beta) \equiv 2^{n/2} \mZ(0, 0; 1, X(1); \beta 2^{-n/4}).
\end{align*}
Since $\varrho$ has the same scaling relation (it corresponds to the $\beta = 0$ case for $\mZ$), we have that
\begin{align*}
\frac{\mZ(C_k^n(X); \beta)}{\varrho(C_k^n(X))} \equiv \frac{\mZ(0, 0; 1, X(1); \beta 2^{-n/4})}{\varrho(0, 0; 1, X(1))}.
\end{align*}
Using again that $X$ and $W$ are independent, and using part (v) of Theorem \ref{Z_properties_theorem}, the distribution of the above variables does not actually depend on $X$. Thus for each $n$ we define independent random variables $A_{i,n}, 1 \leq i \leq 2^n$, each with the distribution
\begin{align*}
1 + A_{i,n} \equiv \frac{\mZ(0, 0; 1, 0; \beta 2^{-n/4})}{\varrho(0, 0; 1, 0)},
\end{align*}
so that
\begin{align*}
M_n^W(X; \beta) \equiv \prod_{i=1}^{2^n} 1 + A_{i,n}.
\end{align*}
As the variables on the right hand side only depend on white noise, the rest of the proof uses only the measure $Q$. We will now show that for every $C > 0$
\begin{align*}
Q \left( \sum_{i=1}^{2^n} \log(1 + A_i^n) > -C \right) \xrightarrow{n \to \infty} 0.
\end{align*}
Combining this with the fact that $M_n^W(X)$ converges almost surely completes the proof.

Observe that $A_{i,n} > -1$ by the fact that $\mZ$ and $\varrho$ are strictly positive, and using Theorem \ref{Z_properties_theorem}, we have that
\begin{align*}
\E_{Q}[A_{i,n}] = 0, \Var_Q(A_{i,n}) = O(\beta^2 2^{-n/2}).
\end{align*}
Let
\begin{align*}
f(x) = x - \frac{x^2}{2} \indicate{-1 < x < 0}
\end{align*}
and define the random variables $B_{i,n}$ by $B_{i,n} = f(A_{i,n})$. Since $f(x) \geq \log(1+x)$, it is sufficient to show that
\begin{align*}
Q \left( \sum_{i=1}^{2^n} B_{i,n} > -C \right) \xrightarrow{n \to \infty} 0.
\end{align*}
The standard Chebyshev bound gives us that
\begin{align*}
Q \left( \sum_{i=1}^{2^n} B_{i,n} > -C \right) \leq \frac{2^n \Var(B_{i,n})}{(C + 2^n \E[B_{i,n}])^2}.
\end{align*}
The proof proceeds by showing that the $B_{i,n}$ have a variance that is on the same order of the magnitude of the $A_{i,n}$, but the mean of the $B_{i,n}$ is negative and of the same order as the standard deviation. To upper bound the variance we use the simple estimate $f(x)^2 \leq 4 x^2$ to get that
\begin{align*}
\Var_{Q}(B_{i,n}) \leq \E_{Q}[B_{i,n}^2] \leq 4 \E_{Q}[ A_{i,n}^2 ] = O(\beta^2 2^{-n/2}).
\end{align*}
For the mean observe that $\E[A_{i,n}] = 0$ implies
\begin{align*}
\E_{Q}[B_{i,n}] = - \tfrac12 \E_{Q}[A_{i,n}^2 \indicate{-1 < A_{i,n} < 0}].
\end{align*}
The chaos expansion \eqref{SHE_chaos_expansion} shows that $A_{i,n} \equiv \sigma 2^{-n/4} Z + Y_n$, where $\sigma$ is some positive constant, $Z$ is a standard normal random variable, and $Y_n$ is a mean zero random variable with $\E[Y_n^2] = 2^{-n}$. In the following Lemma we prove that these three simple facts imply that there is a constant $c > 0$ such that
\begin{align*}
\E_{Q} \left[ A_{i,n}^2 \indicate{-1 < A_{i,n} < 0} \right] \geq c 2^{-n/2}
\end{align*}
for $n$ sufficiently large, and this completes the proof.
\end{proof}

\begin{lemma}
Let $Z$ be a standard normal random variable and $Y_{\epsilon}$ be a collection of random variables such that $\E [ Y_{\epsilon} ] = 0, \E \left[ Y_{\epsilon}^2 \right] = O(\epsilon^4)$. Then there exists a constant $c > 0$ such that as $\epsilon \downarrow 0$
\begin{align*}
\E[ (\epsilon Z + Y_{\epsilon})^2 \indicate{\epsilon Z + Y_{\epsilon} < 0} ] \geq c \epsilon^2.
\end{align*}
\end{lemma}

\begin{proof}
First observe that
\begin{align*}
\E & \left[ (\epsilon Z + Y_{\epsilon})^2 \indicate{\epsilon Z + Y_{\epsilon} < 0} \right] = \epsilon^2/2 + \epsilon^2 \E \left[ Z^2 (\indicate{\epsilon Z + Y_{\epsilon} < 0} - \indicate{\epsilon Z < 0}) \right] \\
&\quad \quad \quad \quad \quad + 2 \epsilon \E \left[ Z Y_{\epsilon} \indicate{\epsilon Z + Y_{\epsilon} < 0} \right] + \E \left[ Y_{\epsilon}^2 \indicate{\epsilon Z + Y_{\epsilon} < 0} \right].
\end{align*}
The last two terms are easy to bound since
\begin{align*}
\E \left[ Y_{\epsilon}^2 \indicate{\epsilon Z + Y_{\epsilon} < 0} \right] \leq \E \left[ Y_{\epsilon}^2 \right] = O(\epsilon^4),
\end{align*}
and
\begin{align*}
\epsilon \, \big | \! \E \left[ Z Y_{\epsilon} \indicate{\epsilon Z + Y_{\epsilon} < 0} \right] \! \big | \leq \epsilon \sqrt{\E[Z^2] \E[Y_{\epsilon}^2] } = O(\epsilon^3).
\end{align*}
The remaining term we break into the two cases $|Y_{\epsilon}| > \epsilon^2 L$ and $|Y_{\epsilon}|^2 \leq \epsilon^2 L$ for a number $L > 0$ to be determined later. In the latter case we have
\begin{align*}
\big | \! \E \left[ Z^2 (\indicate{\epsilon Z + Y_{\epsilon} < 0} - \indicate{\epsilon Z < 0}) \indicate{|Y_{\epsilon}| \leq \epsilon^2 L} \right] \! \big | \leq \E \left[ Z^2 \indicate{|\epsilon Z| \leq \epsilon^2 L}\right] = O(\epsilon L).
\end{align*}
In the former case we have
\begin{align*}
\big | \! \E \left[ Z^2 (\indicate{\epsilon Z + Y_{\epsilon} < 0} - \indicate{\epsilon Z < 0}) \indicate{|Y_{\epsilon}| > \epsilon^2 L}\right] \! \big | \leq \sqrt{\E[Z^4] \mathrm{P}(|X_{\epsilon}| > \epsilon^2 L)} \leq \frac{C}{L},
\end{align*}
for some fixed constant $C > 0$. Combining all these estimates gives
\begin{align*}
\E & \left[ (\epsilon Z + Y_{\epsilon})^2 \indicate{\epsilon Z + Y_{\epsilon} < 0} \right] \geq \epsilon^2/2 + \epsilon^2(C/L + O(\epsilon L)) + O(\epsilon^3).
\end{align*}
Choosing $L$ such that $C/L < 1/2$ completes the proof.
\end{proof}

\section{Generalizations}

\subsection{Point-to-Point Polymers \label{sec:p2p}}

The continuum random polymers described thus far have been point-to-line versions: those with fixed starting points but free endpoints. Point-to-point versions with fixed ending points can be constructed just as easily. Taking $x \in \R$, the finite dimensional distributions of the continuum random polymer ending at $x$ are defined by
\begin{align*}
\Pl_{\beta, x}^W \left( X_{\vt_1} \in d \vx_1, \ldots, X_{\vt_k} \in d \vx_k \right) = \frac{1}{\mZ(0,0;1,x; \beta)} \prod_{j=0}^{k} \mZ(\vt_j, \vx_j; \vt_{j+1}, \vt_{j+1};\beta) \, d\vx_1 \ldots d\vx_k,
\end{align*}
with $(\vt_0, \vx_0) = (0,0)$ and $(\vt_{k+1}, \vx_{k+1}) = (1,x)$. It is easily checked that these have these bridges have the same H\"{o}lder continuity and quadratic variation properties as standard Brownian bridges.

\subsection{Polymers of Different Lengths \label{sec:lengths}}

In this paper we have chosen to concentrate only on directed polymers of length $1$, but the definitions and results are easily generalized to polymers of an arbitrary length. In this section we briefly describe how to do this, and also show how polymers of infinite length can be constructed by taking an appropriate limit.

The extension is simple. We first let the space-time white noise $W$ live on $[0, \infty) \times \R$, and observe that this $W$ can be used to construct the family $\mZ(s, y; t, x; \beta)$ of space-time partition functions for any $0 < s < t < \infty$. Then for fixed $T > 0$ we define the path measure $\Pl_{\beta, T}^W$ on $C([0,T])$ by the same formula as in Definition \ref{continuum_fdds} but with $T$ everywhere replacing $1$. The joint measure $\P_{\beta, T}$ is then defined by specifying the white noise restricted to $[0,T] \times \R$ as the environment marginal and the path measure $\Pl_{\beta, T}^W$ as the conditional.

It is important to observe that, as in the case of discrete polymers, the path measures $\Pl_{\beta, T}^W$ are \textit{not} consistent in $T$. Projecting a path of length $T$ onto a path of a smaller length $T'$ does \textit{not} produce a path distributed according to $\Pl_{\beta, T'}^W$. A mathematical explanation is given by \eqref{continuum_fdds_eqn}, which shows that the transition probabilities from any space-time point are not constant in $T$. The intuitive explanation, however, is that the polymer path always surveys the environment laid out before it and then constructs the transition probabilities accordingly. If the environment is perturbed, in this case by ignoring the part from $T'$ to $T$, then the transition probabilities at \textit{all} space-time points before time $T'$ also change.

We remark, however, that there is a simple scaling relation between the different measures. Suppose $(W,X)$ is distributed according to $\P_{\beta, T}$ for some $T > 0$. Then the rescaled pair $(W_*, X_*)$ defined by
\begin{align*}
W_*(t,x) = T^{-3/4} W(tT, x \sqrt{T}), \quad X_*(t) = T^{-1/2} X(tT),
\end{align*}
has the distribution of $\P_{\beta T^{1/4}}$. This can be seen in many ways, but it is likely easiest to understand by using the Gibbsian formalism \eqref{formal_gibbs_expression}, the white noise scaling relation \eqref{white_noise_scaling}, and the scaling relation of Brownian motion. An immediate consequence of the scaling relation for $(W,X)$ is that
\begin{align*}
\lim_{T \to \infty} \P_{\beta T^{-1/4}, T}
\end{align*}
exists for every $\beta > 0$, meaning that there are continuum polymers of infinite length. The scaling of $\beta$ by $T^{-1/4}$ as $T \to \infty$ is the exact analogue of the scaling used in \cite{AKQ:IDR} on the \textbf{intermediate disorder regime}, where we prove that the transition probabilities of the discrete directed polymer converge those of the continuum one.

\subsection{A Stochastic Differential Equation for the Continuum Polymer \label{sec:SDE}}

Given the environment recall that the continuum polymer evolves in a Markovian way. Since it has continuous paths it is reasonable to expect that there is a stochastic differential equation governing its dynamics. Formally we know what this SDE is, but we do not know precisely how to make sense of it. The diffusion term of the SDE is a standard Brownian motion, but the drift term is highly singular and not an object that we know how to deal with. In this section we give a purely formal description of the SDE.

The SDE we describe is for the point-to-line polymer; the SDE for the point-to-point polymer can be constructed similarly. Let
\begin{align*}
h(s,y) = \log \mZ(s,y; 1, *; \beta).
\end{align*}
Formally speaking, the $\mZ$ field, as a function of $s$ and $y$, satisfies the stochastic PDE
\begin{align*}
\partial_s \mZ = -\tfrac{1}{2} \partial_{yy} \mZ - \beta W, \quad Z(1, y; 1, *) = 1.
\end{align*}
Observe that in $s$ the SPDE is backwards in time. This can be seen by noticing that the role of $s$ and $t$ in \eqref{exponential_shorthand} is reversed, and so comparing with the forward stochastic heat equation \eqref{SHE} we see that the time direction for the SPDE should also be reversed. It is also in perfect analogy with the discrete evolution equation \eqref{discrete_backward_difference_equation} for the $Z(i,x;j,*;\beta)$ field, which is solved backwards in time. Then straightforward computations show that $h$ obeys the KPZ equation \cite{KPZ}
\begin{align}\label{KPZ_eqn}
\partial_s h = -\tfrac{1}{2} \partial_{yy} h - \tfrac{1}{2} \left( \partial_y h \right)^2 - \beta W, \quad h(1,y) = 0.
\end{align}
The initial time condition of $h(1,y) = 0$ seems odd, but notice that the $-\beta W$ term on the right hand side of the KPZ equation immediately perturbs $h$ away from zero. Now by an application of Ito's formula we have that
\begin{align*}
dh(s, X_s) = \partial_s h(s,X_s) \, ds + \partial_y h(s,X_s) \, dX_s + \tfrac{1}{2} \partial_{yy} h(s, X_s) \, ds,
\end{align*}
with the last term being the standard It\^{o} correction. The $ds$ coefficient appears there because the continuum polymer has linearly growing quadratic variation. The KPZ equation \eqref{KPZ_eqn} implies that this can be rewritten as
\begin{align*}
dh(s, X_s) = - \tfrac{1}{2} (\partial_y h(s, X_s))^2 \, ds - \beta W(s, X_s) \, ds + \partial_y h(s, X_s) \, dX_s,
\end{align*}
and after rearranging and integrating each side from $s=0$ to $s=1$ we have
\begin{align*}
\beta \int_0^1 W(s, X_s) \, ds = \int_0^1 \partial_y h(s, X_s) \, dX_s - \frac{1}{2} \int_0^1 \left( \partial_y h(s,X_s) \right)^2 \, ds - h(1, X_1) + h(0,0).
\end{align*}
But $h(1, X_1) = 0$ from the boundary condition and $h(0,0) = \log \mZ(0,0;1,*;\beta)$ by definition of $h$. Exponentiating both sides gives
\begin{align*}
\exp \left \{ \beta \int_0^1 W(s, X_s) \, ds \right \} = \mZ(0,0;1,*;\beta) \exp \left \{ \int_0^1 \partial_y h(s, X_s) \, dX_s - \frac{1}{2} \int_0^1 \left( \partial_y h(s,X_s) \right)^2 \, ds \right \}.
\end{align*}
Using this equation and the Gibbsian formalism \eqref{formal_gibbs_expression}, we can rewrite the path measure as
\begin{align*}
d \Pl_{\beta}^W(X) &= \frac{1}{\mZ(0,0;1,*;\beta)}  \exp  \left \{ \beta \int_0^1 W(s, X_s) \, ds \right \} \, d \! \Pl(X) \\
&= \, \exp \left \{ \int_0^1 \partial_y h(s, X_s) \, dX_s - \frac{1}{2} \int_0^1 \left( \partial_y h(s,X_s) \right)^2 \, ds \right \} \, d \! \Pl (X).
\end{align*}
By Girsanov's theorem, the last expression says that $\Pl_{\beta}^W$ is the probability measure induced on paths by the diffusion
\begin{align*}
d X_s = \partial_y h(s, X_s) \, ds + dB_s,
\end{align*}
where $B$ is a standard Brownian motion.

The main difficulty in interpreting this SDE is that we do not know how to make sense of the spatial derivative of the field $h$. It is well known that $y \mapsto \mZ(s,y;1,*)$ is H\"{o}lder continuous of order $1/4-\epsilon$ for every $\epsilon > 0$, but it is certainly not differentiable. Even if the derivative $\partial_y h$ can be made sense of in an appropriate way it is certainly not in the Cameron-Martin class \cite{Janson:GHS}. This accounts for the singularity of the continuum polymer with respect to Wiener measure.

\bibliographystyle{alpha}
\bibliography{polymers}

\end{document}